\documentclass[11pt]{amsart}

\usepackage[utf8]{inputenc}
\usepackage{hyperref}
\usepackage[english]{babel}
\usepackage{mathtools}
\usepackage{amsmath,amssymb,amsthm,graphicx,pstricks,amscd,amsfonts,latexsym}
\usepackage{graphicx}
\usepackage[textsize=footnotesize,backgroundcolor=white]{todonotes}

\usepackage{tikz-cd}
\usepackage{tasks}

\newtheorem{theorem}{Theorem}[section]
\newtheorem{lemma}[theorem]{Lemma}
\newtheorem{proposition}[theorem]{Proposition}
\newtheorem{corollary}[theorem]{Corollary}

\newtheorem{thm}{Theorem}
\theoremstyle{definition}
\newtheorem{definition}[theorem]{Definition}
\newtheorem{conjecture}[theorem]{Conjecture}
\newtheorem{problem}[theorem]{Problem}

\newtheorem{defthm}[theorem]{Definition-Theorem}
\newtheorem{remark}[theorem]{Remark}
\newtheorem{example}[theorem]{Example}
\newtheorem{notation}[theorem]{Notation}
\newtheorem{convention}[theorem]{Convention}
\newtheorem{question}[theorem]{Question}

\newcommand{\thistheoremname}{}
\newtheorem{genericthm}[theorem]{\thistheoremname}
\newenvironment{namedthm}[1]
{\renewcommand{\thistheoremname}{#1}%
	\begin{genericthm}}
	{\end{genericthm}}
	
\renewcommand{\mod}{\operatorname{mod}}
\newcommand{\proj}{\operatorname{proj}}
\newcommand{\ind}{\operatorname{ind}}
\newcommand{\Hom}{\operatorname{Hom}}
\newcommand{\Alg}{\operatorname{-Alg}}
\newcommand{\Ext}{\operatorname{Ext}}
\newcommand{\ext}{\operatorname{ext}}
\renewcommand{\hom}{\operatorname{hom}}

\newcommand{\twosilt}{\operatorname{2-silt}}

\renewcommand{\dim}{\operatorname{dim}}
\newcommand{\ndo}{\operatorname{end}}
\newcommand{\Ker}{\operatorname{Ker}}
\newcommand{\End}{\operatorname{End}}
\newcommand{\codim}{\operatorname{codim}}

\newcommand{\add}{\operatorname{add}}
\newcommand{\Coker}{\operatorname{Coker}}
\newcommand{\rep}{\operatorname{rep}}
\newcommand{\smd}{\operatorname{smd}}
\newcommand{\Walls}{\operatorname{Walls}}

\newcommand{\GL}{\operatorname{GL}}

\newcommand{\Aut}{\operatorname{Aut}}
\newcommand{\Stab}{\operatorname{Stab}}
\newcommand{\Mat}{\operatorname{Mat}}

\newcommand\scalemath[2]{\scalebox{#1}{\mbox{\ensuremath{\displaystyle #2}}}}

\begin{document}
	\title[]{The non-decreasing condition on $\mathbf{g}$-vectors}
	\author{Mohamad Haerizadeh}
	\email{hyrizadeh@ut.ac.ir}
	\author{Siamak Yassemi}
	\email{ yassemi@ut.ac.ir}
	 
	\date{\today}
	\begin{abstract}
		The non-decreasing condition on g-vectors is introduced.
		Our study shows that this condition is both necessary and sufficient to ensure that the generically indecomposable direct summands of a given g-vector are linearly independent.
		Additionally, we prove that for any finite dimensional algebra $\Lambda$,  under the non-decreasing condition, the number of generically indecomposable irreducible components that appear in the decomposition of a given generically $\tau$-reduced component is lower than or equal to $|\Lambda|$.
		This solves the conjecture concerning the cardinality of component clusters by Cerulli-Labardini-Schröer, in a reasonable generality.
		Lastly, we study numerical criteria to check the wildness of g-vectors.
	\end{abstract}
		\maketitle
		\tableofcontents
	\section*{Introduction}
		Because of technical reasons concerning the Voigt's Isomorphism, to generalize results about general representations of quivers to bound quivers, one should to work with general representations of the irreducible components of varieties of representations.
		In this regard, inspired by the works of Kac \cite{Ka82} and Schofield \cite{Sch92}, Crawley-Boevey and Schr\"{o}er \cite{CBSc02} construct decomposition for irreducible components of varieties of representation, which is an analogue of the Krull-Remak-Schmidt decomposition.
		
		There also exists a parallel approach that relies on general presentations of varieties of presentations. In their paper \cite{DeFe15}, Derksen and Fei show that general presentations present general representations \cite[Theorem 2.3]{DeFe15}. Moreover, they introduce decomposition for varieties of presentations, which is nicely related to decomposition of general representations. Indeed, they introduce canonical decomposition of g-vectors \cite[Theorem 4.4]{DeFe15}, which we currently call generic decomposition.
		
		The main results of our study concern the generic decomposition of g-vectors and the decomposition of generically $\tau$-reduced components. More precisely, we generalize some results of \cite{CILFS14} and \cite{AsIy23}, about g-vectors and generically $\tau$-reduced components by applying results of the papers \cite{Pla13}.
		In this paper, Plamondon assigns a generically $\tau$-reduced component to each g-vector. The process gives us a surjective map from the Grothendieck group to the set of generically $\tau$-reduced components of $\rep(\Lambda)$ \ref{thm5500}.
		
		On the other hand, Cerulli-Irelli, Labardini-Fragoso and Schr\"{o}er study the decompositions of generically $\tau$-reduced components \cite[Theorem 1.3]{CILFS14}. This theorem together with Lemma \ref{lem6100} and Lemma \ref{lem5200} reveals that there is a link between generic decomposition of g-vectors and decomposition of generically $\tau$-reduced components.
		\begin{thm}[\ref{thm6200}]
			Let $\{\mathcal{Z}_1, \mathcal{Z}_2,\cdots, \mathcal{Z}_s\}$ be a set of irreducible components of $\rep(\Lambda)$. Then the following statements are equivalent.
			\begin{enumerate}
				\item[$(1)$] $\mathcal{Z}=\overline{\mathcal{Z}_{1}\oplus\mathcal{Z}_{2}\oplus\cdots\oplus \mathcal{Z}_{s}}$ is a generically $\tau$-reduced component.
				\item[$(2)$] Each $\mathcal{Z}_i$ is a generically $\tau$-reduced component and
				$$\min\hom_{\Lambda}(\mathcal{Z}_i,\tau\mathcal{Z}_j)=0,$$
				for all $1\le i\neq j\le s$.
				\item[$(3)$] Each $\mathcal{Z}_i$ is a generically $\tau$-reduced component and
				$$e(g^{\mathcal{Z}_i},g^{\mathcal{Z}_j})=0,$$
				for all $1\le i\neq j\le s$.
			\end{enumerate}
			In this case, $g^{\mathcal{Z}}=g^{\mathcal{Z}_1}\oplus\cdots\oplus g^{\mathcal{Z}_s}$ is the generic decomposition of $g^{\mathcal{Z}}$.
		\end{thm}
		This connection leads us to provide an answer for the conjecture \ref{con5900}, by Cerulli-Labardini-Schröer concerning component clusters and thus to generalize \cite[Theorem 6.1]{CILFS14}.
		\begin{thm}[\ref{thm6500}]
			Let $\mathcal{Z}$ be a generically $\tau$-reduced component. If $g^\mathcal{Z}$ satisfies the non-decreasing condition (See Definition \ref{def4400}), then $|\mathcal{Z}|\le |\Lambda|$. In particular,
			if $\Lambda$ satisfies the non-decreasing condition, then for all generically $\tau$-reduced components $\mathcal{Z}$, we have $|\mathcal{Z}|\le|\Lambda|$.
		\end{thm}
		\begin{thm}[\ref{cor6650}]
			Let $\mathcal{Z}$ be a generically $\tau$-reduced component. Then the following statements are equivalent.
			\begin{itemize}
				\item[$(1)$] $g^\mathcal{Z}\in\mathcal{C}^{\circ}(\mathrm{P})$ for a $2$-term silting complex $\mathrm{P}\in K^b(\proj\Lambda)$.
				\item[$(2)$] $|\mathcal{Z}|=|\Lambda|$ and $g^\mathcal{Z}$ satisfies the non-decreasing condition.
			\end{itemize}
		\end{thm}
		More precisely, we show that under the non-decreasing condition, the generically indecomposable direct summands of a given g-vector are linearly independent. Then we apply the connecting Theorem \ref{thm6200}.
		Furthermore, we prove that generically indecomposable direct summands in any generic decomposition are linearly independent, exactly, when $\Lambda$ satisfies the non-decreasing condition. Due to Proposition \ref{prop4450}, the following theorem also generalizes \cite[Proposition 6.7(b)]{AsIy23}.
		\begin{thm}[\ref{thm4700}]
			Let $g$ be a g-vector. If $g$ satisfies the non-decreasing condition, then $\ind(g)$ is linearly independent. Moreover, the following statements are equivalent.
			\begin{enumerate}
				\item[$(1)$] $\Lambda$ satisfies the non-decreasing condition.
				\item[$(2)$] $\ind(g)$ is linearly independent, for each g-vector $g$.
			\end{enumerate}
		\end{thm}
		Since determining whether a g-vector is tame or wild may be a little complicated, examining some numerical criteria can be helpful. Therefore, we will study some criteria in the last section.
		\begin{thm}[\ref{thm7000}]
			Consider a g-vector $g$ with no negative\footnote{We say $g$ is positive (negative, respectively) if $[P^{g-}]=0$ ($[P^{g+}]=0$, respectively).} direct summand. Then $g$ is wild if and only if
			$$\min\hom_{\Lambda}(\mathcal{Z}_g,\mathcal{Z}_g)>\langle g,d(g)\rangle.$$
		\end{thm}
	\section*{Conventions}
		We consider a basic finite dimensional algebra $\Lambda$ over an algebraically closed field $k$. So because of Morita equivalence, one can assume $\Lambda=kQ/I$, where $(Q,I)$ is a bound quiver with $n$ vertices. For simplicity, we assume that $Q$ is connected. It is well known that, up to isomorphism, there are exactly $n$ simple modules $\{S_{(1)},S_{(2)},\cdots,S_{(n)}\}$. Moreover, the Grothendieck group of the category of finitely generated modules $K_{0}(\mod\Lambda)$ is a free abelian group of finite rank $n$ with a basis  $\{[S_{(1)}],[S_{(2)}],\cdots,[S_{(n)}]\}$. One can also consider the Grothendieck group of category of finitely generated projective modules $K_0(\proj\Lambda)$ which is also a free abelian group of finite rank $n$ with a basis $\{[P_{(1)}],[P_{(2)}],\cdots,[P_{(n)}]\}$, where $P_{(i)}$ is a projective cover of $S_{(i)}$, for each $i=1,2,\cdots,n$. We also fix the set $\{e_1,\cdots,e_n\}$ as the complete set of primitive
		orthogonal idempotents corresponding to $\{P_{(1)},\cdots,P_{(n)}\}$.
	\section{Preliminaries}
		In this section, we review some geometric methods in representation theory of finite dimensional algebras. Nonetheless, we refer non-expert readers to the  Crawley-Boevey's note \cite{CB20}.
		\begin{convention}
			All modules will be finitely generated left modules unless stated otherwise. As well, product of arrows in a quiver algebra is something like the composition of functions.
		\end{convention}
		\begin{definition}\label{def100}
			Let $M$ be a finitely generated module. Then the \textit{dimension vector} of $M$ is $(\dim_k e_1M,\cdots,\dim_k e_nM)$. The category of representations (modules, respectively) of dimension vector $d$ is denoted by $\rep_d(\Lambda)$ ($\mod_d(\Lambda)$, respectively).
		\end{definition}
		\begin{convention}
			Let $d$ be a dimension vector.
			$\rep_d(\Lambda)$ and $\mod_d(\Lambda)$ are naturally equivalent. So we use them freely instead of each other. Indeed, we may consider a module as a representation and vice versa.
		\end{convention}
		\begin{notation}
			For modules $X$ and $Y$, we set
			\begin{tasks}[style=itemize](2)
				\task $\hom_{\Lambda}(X,Y)=\dim_k\Hom_{\Lambda}(X,Y)$
				\task $\ndo_{\Lambda}(X)=\dim_k\End_{\Lambda}(X)$
				\task $\ext^1_{\Lambda}(X,Y)=\dim_k\Ext^1_{\Lambda}(X,Y)$
				\task $\ext^1_{\Lambda}(X)=\dim_k\Ext^1_{\Lambda}(X,X)$
			\end{tasks} 
		\end{notation}
		\begin{definition}\label{def200}
			Let $\mathcal{X}$ be a topological space. A function $f:\mathcal{X}\rightarrow\mathbb{Z}$ is \textit{upper semi-continuous}, if for each $t\in\mathbb{Z}$ the subset
			\[\{x\in\mathcal{X}\mid f(x) < t\}\]
			is open in $\mathcal{X}$.
		\end{definition}
		\begin{namedthm}{Varieties of representations}
			For a dimension vector $d$,
			\[\rep_d(Q)=\prod_{\alpha\in Q_1}\Hom_{k}(k^{d_{s(\alpha)}},k^{d_{t(\alpha)}})\footnote{For an arrow $\alpha\in Q_1$, $t(\alpha)$ (respectively, $s(\alpha)$) is the vertex at the target of $\alpha$ (respectively, the source one).}\]
			is an affine variety whose points correspond to representations of $Q$.
			Moreover, $\rep_d(\Lambda)$ is a Zariski-closed sub-variety of $\rep_d(Q)$ that satisfies the relations in the ideal $I$.\\
			As well, the space of representations of dimension $d\in\mathbb{N}$
			\[\rep_{d}(\Lambda)=\Hom_{k\Alg}(\Lambda,\Mat_{d}(k))\]
			is an affine variety.
		\end{namedthm}
		\begin{lemma}\label{lem300}
			Consider dimension vectors (or any positive integers) $d$ and $d'$.
			The following functions are all upper semi-continuous.
			\begin{itemize}
				\item $\hom_{\Lambda}(-,?), \ext_{\Lambda}^1(-,?):\rep_d(\Lambda)\times\rep_{d'}(\Lambda)\rightarrow\mathbb{Z}$
				\item $\ndo_{\Lambda}(-), \ext_{\Lambda}^1(-):\rep_d(\Lambda)\rightarrow\mathbb{Z}$
			\end{itemize}
		\end{lemma}
		\begin{proof}
			See \cite[Lemma 4.2]{CBSc02}.
		\end{proof}
		\begin{remark}\label{rem400}
			Each of the upper semi-continuous functions in the above example admits a minimum value on given sub-varieties $\mathcal{X}\subseteq\rep_d(\Lambda)$ and $\mathcal{Y}\subseteq\rep_{d'}(\Lambda)$.
			\begin{tasks}[style=itemize]
				\task $\min\hom_{\Lambda}(\mathcal{X},\mathcal{Y}):=\min\{\hom_{\Lambda}(X,Y)\mid X\in\mathcal{X}, y\in\mathcal{Y}\}$
				\task $\min\ext^1_{\Lambda}(\mathcal{X},\mathcal{Y}):=\min\{\ext^1_{\Lambda}(X,Y)\mid X\in\mathcal{X}, y\in\mathcal{Y}\}$
				\task $\min\ndo_{\Lambda}(\mathcal{X}):=\min\{\ndo_{\Lambda}(X)\mid X\in\mathcal{X}\}$
				\task $\min\ext^1_{\Lambda}(\mathcal{X}):=\min\{\ext^1_{\Lambda}(X)\mid X\in\mathcal{X}\}$
			\end{tasks}
			So $\{(X,Y)\in\mathcal{X}\times\mathcal{Y}\mid\hom_{\Lambda}(X,Y)=\min\hom_{\Lambda}(\mathcal{X},\mathcal{Y})\}$ is an open subset of $\mathcal{X}\times\mathcal{Y}$. The similar statements hold for the other upper semi-continuous functions.
		\end{remark}
		\begin{definition}\label{def500}
			Let $\mathcal{X}$ be a variety.  Then ``\textit{a general element} of $\mathcal{X}$'' means an arbitrary element of a dense open subset of $\mathcal{X}$. For example, one can say that for a general module $M\in\rep_d(\Lambda)$, $\ndo_{\Lambda}(M)=\min\ndo_{\Lambda}(\mathcal{X})$.
		\end{definition}
		\begin{definition}\label{def600}
			A non-empty topological space $\mathcal{Z}$ is said to be \textit{irreducible}, if it can not be written as a union of two closed subsets, that is, $\mathcal{Z}=\mathcal{C}_1\cup\mathcal{C}_2$ implies that $\mathcal{Z}=\mathcal{C}_1$ or $\mathcal{Z}=\mathcal{C}_2$. Equivalently, $\mathcal{Z}$ is non-empty and any non-empty open subset of $\mathcal{Z}$ is dense.
		\end{definition}
		\begin{remark}\label{rem700}
			Let $\mathcal{Y}$ be an irreducible subset of a variety $\mathcal{X}$. Then the following statements hold.
			\begin{itemize}
				\item $\overline{\mathcal{Y}}$ is irreducible.
				\item Any non-empty open subset of $\mathcal{Y}$ is irreducible.
				\item For a morphism of varieties $f:\mathcal{X}\rightarrow\mathcal{X}'$, $f(\mathcal{Y})$ is irreducible.
			\end{itemize}
		\end{remark}
		\begin{defthm}\label{defthm800}
			A variety $\mathcal{X}$ can be written as a finite union of maximal irreducible (closed) subsets, in a unique way. These maximal irreducible (closed) subsets are called \textit{irreducible components}.
		\end{defthm}
		\begin{remark}\label{rem900}
			Let $\mathcal{X}$ be a variety. Then any irreducible subset of $\mathcal{X}$ is contained in a (not necessarily unique) irreducible component.
		\end{remark}
		\begin{definition}\label{def910}
			Let $f:\mathcal{X}\rightarrow\mathcal{Y}$ is a morphism of varieties. Then $f$ is said to be \textit{dominant}, if the image is dense in $\mathcal{Y}$.
		\end{definition}
		\begin{remark}\label{rem920}
			By definition, a composition of dominant morphisms is dominant as well.
		\end{remark}
		\begin{definition}\label{def1000}
			Let $\mathcal{X}$ be a variety. A \textit{constructible subset} is a finite union of locally closed subsets.
		\end{definition}
		\begin{lemma}\label{lem1100}
			Let $\mathcal{X}$ be an irreducible variety and $\mathcal{Y}$ a constructible subset. If $\overline{\mathcal{Y}}=\mathcal{X}$, then there is an open subset $\mathcal{U}$ of $\mathcal{X}$ such that $\mathcal{U}\subseteq\mathcal{Y}$.
		\end{lemma}
		\begin{namedthm}{Chevalley's Constructible Theorem}\label{thm1200}
			Let $f:\mathcal{X}\rightarrow\mathcal{Y}$ be a morphism of varieties. Then $f(\mathcal{X})$ is constructible. Therefore, if $\mathcal{Y}$ is irreducible and $f$ is dominant, then a general element of $\mathcal{Y}$ is the image of a general element of $\mathcal{X}$.
		\end{namedthm}
		\begin{definition}\label{def1300}
			Let $\mathcal{X}$ be a variety. Then the \textit{dimension} of $\mathcal{X}$ is defined as follows.
			$$\scalemath{0.9}{\dim\mathcal{X}:=\sup\{l\in\mathbb{N}\mid\textit{$\exists\mathcal{Z}_0\subset\mathcal{Z}_1\subset\cdots\subset\mathcal{Z}_l$ of irreducible closed subsets of $\mathcal{X}$}\}}$$
		\end{definition}
		\begin{proposition}\label{prop1400}
			Let $\mathcal{X}$ be a variety and $\mathcal{Y}\subseteq\mathcal{X}$ be locally closed. Then the following statements hold.
			\begin{itemize}
				\item $\mathcal{X}$ is of finite dimension.
				\item $\dim\mathcal{Y}\le\dim\mathcal{X}$. Especially, if $\mathcal{Y}$ is proper closed and $\mathcal{X}$ is irreducible, then $\dim\mathcal{Y}<\dim\mathcal{X}$.
				\item Let $\mathcal{X}$ be irreducible. Then $\dim\mathcal{X}=\dim\mathcal{Y}$ if and only if $\overline{\mathcal{Y}}=\mathcal{X}$. In this case, $\mathcal{Y}$ is open (and thus dense).
			\end{itemize}
		\end{proposition}
		\begin{remark}\label{rem1500}
			A finite dimensional vector space $\mathcal{X}$ over $k$ is an irreducible affine variety and $\dim\mathcal{X}=\dim_k\mathcal{X}$. In particular, for modules $X$ and $Y$, $\Hom_{\Lambda}(X,Y)$ is an irreducible affine variety of dimension $\hom_{\Lambda}(X,Y)$.
		\end{remark}
		\begin{definition}\label{def1600}
			For $l\in\mathbb{N}$, consider the variety of $l\times l$ matrices $\Mat_l(k)$. The \textit{general linear group} $\GL_l(k):=\{A\in\Mat_l(k)\mid \det A\neq 0\}$ is an open sub-variety of $\Mat_l(k)$. Moreover, for a dimension vector $d=(d_1,\cdots,d_n)$, we define an affine algebraic group\footnote{An \textit{algebraic group} is a group which is also a variety with the property that the multiplication and inversion maps are morphisms of varieties.}
			$$\GL_d(k):=\prod_{i=1}^n\GL_{d_i}(k).$$
		\end{definition}
		\begin{namedthm}{Actions by conjugation}\label{def1800}
			Consider a dimension vector $d$.
			$\GL_d(k)$ acts on $\rep_d(\Lambda)$ by conjugation as follows:\\ For $\mathfrak{g}=(\mathfrak{g}_1,\cdots,\mathfrak{g}_n)\in\GL_d(k)$ and $M\in\rep_d(\Lambda)$,
			\begin{itemize}
				\item $(\mathfrak{g}\cdot M)_i=M_i$, for each vertex $i$ of $Q$, and
				\item $(\mathfrak{g}\cdot M)_{\alpha}=\mathfrak{g}_{t(\alpha)}\circ M_{\alpha}\circ\mathfrak{g}_{s(\alpha)}^{-1}$, for each arrow $\alpha$ of $Q$.
			\end{itemize}
			Now, consider representations $M$ and $M'$ with dimension vectors $d$ and $d'$, respectively. $\GL_d(k)\times\GL_{d'}(k)$ acts on $\Hom_{\Lambda}(M,M')$ by conjugation as follows:\\
			Consider $\mathfrak{g}=(\mathfrak{g}_1,\cdots,\mathfrak{g}_n)\in\GL_d(k)$, $\mathfrak{g}'=(\mathfrak{g}_1',\cdots,\mathfrak{g}_n')\in\GL_{d'}(k)$ and $a=(a_1,\cdots,a_n)\in\Hom_{\Lambda}(M,M')$. Then we define
			$$((\mathfrak{g},\mathfrak{g}')\cdot a)_i:=\mathfrak{g'}_i\circ a_i\circ\mathfrak{g}_i^{-1}$$
			for each $i\in Q_0$.
		\end{namedthm}
		\begin{namedthm}{Orbits}\label{prop1900}
			Let $\mathfrak{G}$ be an (affine) algebraic group and $\mathcal{X}$ be a variety. Assume that $\mathfrak{G}$ acts on $\mathcal{X}$. For $x\in\mathcal{X}$, the following statements hold.
			\begin{itemize}
				\item The orbit $\mathcal{O}_x$ is a locally closed and irreducible subset of $\mathcal{X}$.
				\item The stabilizer $\Stab_{\mathfrak{G}}(x):=\{\mathfrak{g}\in \mathfrak{G}\mid \mathfrak{g}\cdot x=x\}$ is a closed subset of $\mathfrak{G}$.
				\item $\dim\mathcal{O}_x=\dim \mathfrak{G}-\dim\Stab_{\mathfrak{G}}(x)$.
			\end{itemize}
			Consider a dimension vector $d=(d_1,\cdots,d_n)$ and the action of $\GL_d(k)$ on $\rep_d(\Lambda)$ by conjugation. Then for a representation $M\in\rep_d(\Lambda)$,
			$$\dim\mathcal{O}_M=\dim\GL_d(k)-\dim\Aut_{\Lambda}(M).$$
			But $\Aut_{\Lambda}(M)=\GL_{\mathbf{d}}(k)\cap\End_{\Lambda}(M)\subseteq\Mat_{\mathbf{d}}(k)$, where $\mathbf{d}=\Sigma_{i=1}^n d_i$. Hence $\Aut_{\Lambda}(M)$ is an open sub-variety of $\End_{\Lambda}(M)$. Thus $\dim\Aut_{\Lambda}(M)=\ndo_{\Lambda}(M)$. Therefore,
			$$\dim\mathcal{O}_M=\dim\GL_d(k)-\ndo_{\Lambda}(M)=\langle d,d \rangle -\ndo_{\Lambda}(M).$$
			So by upper semi-continuity of $\ndo_{\Lambda}(-)$, it is easy to see that $$\codim\mathcal{O}_{-}:\rep_d(\Lambda)\rightarrow\mathbb{Z}$$
			is upper semi-continuous. Immediately, for an irreducible component $\mathcal{Z}$ of $\rep_d(\Lambda)$,
			$$\{X\in\mathcal{Z}\mid\codim_{\mathcal{Z}}\mathcal{O}_X=\min_{M\in\mathcal{Z}}\codim_{\mathcal{Z}}\mathcal{O}_M\}$$
			is an open dense subset of $\mathcal{Z}$.
		\end{namedthm}
		\begin{theorem}\label{thm2100}
			Consider a dimension vector (or any positive integer) $d$.
			All irreducible components of $\rep_d(\Lambda)$ are $\GL_d(k)$-stable, that is, they are closed under the action of $\GL_d(k)$ on $\rep_d(\Lambda)$.
		\end{theorem}
		\begin{proof}
			Consider an irreducible component $\mathcal{Z}$ of $\rep_d(\Lambda)$.
			Since $\GL_d(k)$ and $\mathcal{Z}$ are irreducible, the image of the morphism $\GL_d(k)\times\mathcal{Z}\rightarrow\rep_d(\Lambda)$ is irreducible, as well. So it must be contained in an irreducible component of $\rep_d(\Lambda)$. Therefore, $\GL_d(k)\cdot\mathcal{Z}=\mathcal{Z}$.
		\end{proof}
		\begin{corollary}\label{cor2200}
			Consider a dimension vector (or any positive integer) $d$. $\rep_d(\Lambda)$ is a union of  finite number of $\GL_d(k)$-stable irreducible components.
		\end{corollary}
	\section{$\mathrm{g}$-vectors}
		In this section, we introduce the non-decreasing condition on g-vectors and finite dimensional algebras. We prove that under the non-decreasing condition on $g$, the number of the generically indecomposable direct summands of $g$ never decreases, after multiplications.
		Then we will see that if $g$ satisfies the non-decreasing condition, then $\ind(g)$ is linearly independent. Moreover, we show that the converse of the above statements also hold, when $\ind(tg)$ is linearly independent, for each $t\in\mathbb{N}$.
		\begin{convention}
			To simplify some expressions, in this paper, we always assume that $K_0(\proj\Lambda)=K_0(K^b(\proj\Lambda))$. By abuse of notation, the elements of $K_0(\mod\Lambda)_\mathbb{R}$ and $K_0(\proj\Lambda)_\mathbb{R}$  are also assumed to be vectors in $\mathbb{R}^n$.
		\end{convention}
		\begin{definition}\label{def2300}
			An element of $K_0(\proj\Lambda)$ is called a \textit{g-vector}.
			Consider a complex $\mathrm{P}=\cdots\rightarrow P^{-1}\rightarrow P^0\rightarrow P^1\rightarrow\cdots\in K^b(\proj\Lambda)$. We set
			$$g^{\mathrm{P}}:=\sum_{i\in\mathbb{Z}}(-1)^i[P^i]\in K_0(\proj\Lambda)\cong\mathbb{Z}^n$$
			as the g-vector of $\mathrm{P}$.
			Moreover, if $\mathrm{P}=P^{-1}\rightarrow P^0$ is a minimal projective presentation of a module $M$, then we assign a g-vector to $M$ as follows: $$g^M:=g^{\mathrm{P}}\in K_0(\proj\Lambda)$$
		\end{definition}
		\begin{remark}\label{rem2400}
			Let $M$ be a module and
			$g^M=(g_1 , g_2 , \cdots , g_n)$. Then for each $1\le i\le n$,
			$$g_i=\hom_{\Lambda}(M,S_{(i)})-\ext^1_{\Lambda}(M,S_{(i)}).$$
			Moreover, by \cite[Corollary 1.3]{GLFS23}, $-g_i:\rep_d(\Lambda)\rightarrow\mathbb{Z}$ is upper semi-continuous, for any dimension vector (or any positive integer) $d$.
		\end{remark}
		\begin{convention}
			For projective modules $P^
			{-1}$ and $P^0$, one can consider each morphism in $\Hom_\Lambda(P^{-1},P^0)$ as a two term complex in $K^b(\proj\Lambda)$.
		\end{convention}
		\begin{remark}\label{rem2500}
			Let $g$ be a g-vector. Then there exist unique finitely generated projective modules $P^{g+}$ and $P^{g-}$ without any non-zero common direct summands such that $g=[P^{g+}]-[P^{g-}]$.
		\end{remark}
		\begin{notation}
			Let $g$ be a g-vector. $\Hom_\Lambda(g)$ stands for $\Hom_\Lambda(P^{g-},P^{g+})$.
		\end{notation}
		E-invariants firstly were defined for finite dimensional decorated representations of Jacobian algebras \cite{DWZ10} to solve several conjectures in cluster theory. Then it was generalized for finite dimensional modules in \cite{AIR14} and used to generalize some results of the preceding paper. As well, in order to construct a decomposition of presentation spaces, Derksen and Fei \cite{DeFe15} introduced E-invariants for g-vectors.
		\begin{definition}\label{def2600}
			Let $g$ and $h$ be g-vectors. The \textit{E-invariant} of a pair $(a,b)\in\Hom_{\Lambda}(g)\times\Hom_{\Lambda}(h)$ is
			$$e(a,b):=\hom_{K^b(\proj\Lambda)}(a,b[1]).$$
			Moreover, the E-invariant of the pair $(g,h)$ is
			$$e(g,h):=\min\{e(a,b)\mid a\in\Hom_{\Lambda}(g), b\in\Hom_{\Lambda}(h)\}.$$
			Since $e(?,-):\Hom(g)\times\Hom(h)\rightarrow\mathbb{Z}$ is upper semi-continuous \cite[Corollary 3.7]{DeFe15}, $e(g,h)=e(a,b)$ for a general element $(a,b)\in\Hom(g)\times\Hom(h)$.
		\end{definition}
		\begin{namedthm}{Direct sums of g-vectors}\label{defthm2700}
			Consider g-vectors $g$ and $h$. Then $e(g,h)=0=e(h,g)$, if and only if a general element $a\in\Hom_{\Lambda}(g+h)$ can be written as a direct sum $a=b\oplus c$, where $b\in\Hom_{\Lambda}(g)$ and $c\in\Hom_{\Lambda}(h)$.
			In this case, we say $g$ and $h$ are direct summands of $g+h$, and denote $g+h=g\oplus h$.
		\end{namedthm}
		\begin{proof}
			It follows from \cite[Corollary 4.2]{DeFe15}.
		\end{proof}
		\begin{definition}\label{def2800}
			Let $g$ be a g-vector. It is said to be \textit{tame}, if $2g=g\oplus g$. Otherwise, it is called \textit{wild}.
		\end{definition}
		It is easy to see that g-vectors of $\tau$-rigid modules are tame. Moreover, by \cite[Proof of Theorem 3.8]{PYK23}, if $\Lambda$ is tame, then it is \textit{E-tame}, that is, all g-vectors are tame \cite[Definition 6.3]{AsIy23}.
		\begin{remark}\label{rem2900}
			Let $g$ and $h$ be g-vectors.
			If $g+h=g\oplus h$, then so is $tg+sh=tg\oplus sh$, for $t,s\in\mathbb{N}$. In particular, if $g$ is tame, then so is $tg$ for all $t\in\mathbb{N}$.
		\end{remark}
		\begin{remark}\label{rem3000}
			Let $g$ and $h$ be g-vectors. If $g$ is positive or $h$ negative, then $e(g,h)=0$. In particular,
			negative and positive g-vectors are tame.
		\end{remark}
		\begin{proposition}\label{prop3100}
			Consider g-vectors $g$ and $h$. Then $g+h=g\oplus h$ if and only if there exist $a\in\Hom_{\Lambda}(g)$ and $b\in\Hom_{\Lambda}(h)$ such that $e(a,b)=0=e(b,a)$.
			In particular, $g$ is tame if and only if there exist non-zero $a,a'\in\Hom(g)$ with $e(a,a')=0=e(a',a)$.
		\end{proposition}
		\begin{proof}
			Assume that $g+h=g\oplus h$. Then one can find $a''\in\Hom_{\Lambda}(g)$ and $b\in\Hom_{\Lambda}(h)$ with $e(b,a'')=0$.
			By \cite[Lemma 3.10]{AsIy23}, the subsets
			$$\mathcal{U}=\{a'\in\Hom_{\Lambda}(g)\mid\Coker(b)\in\prescript{\perp}{}{\Ker(\nu a')}\}\footnote{$\nu=D\Hom_{\Lambda}(-,\Lambda):\rep(\Lambda)\rightarrow\rep(\Lambda)$},$$
			$$\mathcal{V}=\{a'\in\Hom_{\Lambda}(g)\mid\Ker(\nu b)\in \Coker(a')^{\perp}\}$$
			 of $\Hom_{\Lambda}(g)$ are open.
			Since $\Hom_{\Lambda}(g)$ is irreducible, there is $a\in \mathcal{U}\cap \mathcal{V}$. So by \cite[Proposition 3.11]{AsIy23}, $e(a,b)=0=e(b,a)$.
		\end{proof}
		\begin{remark}\label{rem3200}
			Assume that for g-vectors $g$ and $g'$, $g+g'=g\oplus g'$. Then for a g-vector $h$, we have
			$$e(g\oplus g',h)=e(g,h)+e(g',h),$$
			$$e(h,g\oplus g')=e(h,g)+e(h,g').$$
			Therefore, one can see that $g\oplus g'$ is tame if and only if both of $g$ and $g'$ are tame.
		\end{remark}
		\begin{definition}\label{def3300}
			Consider a g-vector $g$. It is said to be \textit{generically indecomposasble}, if a general element of $\Hom_{\Lambda}(g)$ is indecomposable.
		\end{definition}
		\begin{namedthm}{Generic decompositions of g-vectors}\label{defthm3400}
			Let $g$ be a g-vector. Then there exist unique (up to reordering) generically indecomposable g-vectors $g_1,g_2,\cdots,g_s$ such that
			$$g=g_1\oplus g_2\oplus\cdots\oplus g_s.$$
			This decomposition is called the \textit{generic decomposition} of $g$.
		\end{namedthm}
		\begin{proof}
			See \cite[Theorem 4.4]{DeFe15} or \cite[Theorem 2.7]{Pla13}.
		\end{proof}
		\begin{remark}\label{rem3500}
			By the above assumption,
			for a general $a\in\Hom_{\Lambda}(g)$, there exist general (indecomposable) projective presentations $a_i\in\Hom_{\Lambda}(g_i), 1\le i\le s$, such that $a=a_1\oplus\cdots\oplus a_s$.
		\end{remark}
		To study algebras through g-vectors, we concentrate on the building blocks of g-vectors, namely, generically indecomposable direct summands. So in the rest of this section, we study some conditions on g-vectors and generic decompositions of g-vectors.
		\begin{notation}\label{not3600}
			Let $g$ be a g-vector. The set of generically indecomposable direct summands of $g$ is denoted by $\ind(g)$. As well,  we set $\smd(g)$ (respectively, $\add(g)$) for the set of all direct summands (respectively, the set of all direct summands of direct sums) of $g$.
		\end{notation}
		\begin{remark}\label{rem3700}
			Let $g$ be a tame g-vector and $g=g_1\oplus g_2\oplus\cdots\oplus g_s$ be its generic decomposition. Then it is easy to see that $tg=g_1^{\oplus t}\oplus g_2^{\oplus t}\oplus\cdots\oplus g_s^{\oplus t}$ is also the generic decomposition of $tg$, for all $t\in\mathbb{N}$. Therefore, $|\ind(g)|=|\ind(tg)|$, for all $t\in\mathbb{N}$. Furthermore, by uniqueness of generic decompositions, $\{g_1,\cdots,g_s\}$ is linearly independent \cite[Proposition 6.7(b)]{AsIy23}.
		\end{remark}
		\begin{remark}\label{rem3750}
			Let $\mathrm{P}\in K^b(\proj\Lambda)$ be a $2$-term pre-silting complex. Then we set $\mathcal{C}^\circ(\mathrm{P})$ (respectively, $\mathcal{C}(\mathrm{P})$), for the open cone (respectively, closed cone) generated by the generically indecomposable direct summands of $g^{\mathrm{P}}$.
			Now, consider the wall and chamber structure of $\Lambda$ defined in \cite[Definition 3.3]{BST19}. By \cite[Theorem 3.17]{As21},
			\[K_0(\proj\Lambda)_{\mathbb{Q}}\setminus\Walls=\coprod_{\scalemath{0.6}{\mathrm{P}\in\twosilt\Lambda}}\mathcal{C}^{\circ}(\mathrm{P})_{\mathbb{Q}}.\footnote{$\twosilt\Lambda$ stands for the set of all 2-term silting complexes in $K^b(\proj\Lambda)$.}\]
			\begin{itemize}
				\item If a g-vector $g\in\mathcal{C}(\mathrm{P})$ for some $2$-term silting complex $\mathrm{P}$, then according to \cite[Theorem 2.27]{AiIy12}, it is tame and thus $\ind(g)$ is linearly independent. Moreover, if $g\in\mathcal{C}^{\circ}(\mathrm{P})$, then $|\ind(g)|=n$.
				\item In the case that
				\[g\in\Walls\setminus\coprod_{\scalemath{0.6}{\mathrm{P}\in\twosilt\Lambda}}\mathcal{C}(\mathrm{P})\]
				we don't know if $\ind(g)$ is linearly independent, in general. In addition, it is probable for $g$ to lose wildness or even it may happen that $|\ind(tg)|\neq|\ind(g)|$, for some $t\in\mathbb{N}$.
				In \cite[Question 4.7]{DeFe15}, it was asked that if there is a generically indecomposable wild g-vector $g$ such that $tg$ is not so, for some $t\in\mathbb{N}$.
			\end{itemize}
			Hence, it is reasonable to study about the following questions.
			\begin{enumerate}
				\item Is there any generically indecomposable wild g-vector $g$ such that $tg$ is not indecomposable, for some $t\in\mathbb{N}$?
				\item Is there any wild g-vector $g$ such that $tg$ is tame, for some $t\in\mathbb{N}$?  
			\end{enumerate}
			For the first question, in \cite[Example 5.9]{AsIy23}, there is an example of generically indecomposable wild g-vector $g$ such that $tg$ is not indecomposable for all integers $t\ge 2$. So $|\ind(g)|<|\ind(tg)|$, for all $t\ge 2$. But there is no answer for the second question, up to our knowledge.
		\end{remark}
		\begin{definition}\label{def3800}
			Let $g$ be a g-vector. We say that
			\begin{itemize}
				\item \cite[Definition 5.1]{AsIy23} $g$ satisfies the \textit{ray condition}, if $th$ is generically indecomposable (thus wild), for all $t\in\mathbb{N}$ and wild $h\in\ind(g)$.
				\item $g$ satisfies the \textit{regularity condition}, if $th$ is wild, for all $t\in\mathbb{N}$ and wild $h\in\ind(g)$.
			\end{itemize}
			We say $\Lambda$ satisfies the ray condition (regularity condition, respectively) if all g-vectors satisfy the ray condition (regularity condition, respectively).
		\end{definition}
		\begin{remark}\label{rem3900}
			Let $g$ be a g-vector. Then the following statements are equivalent.
			\begin{enumerate}
				\item[$(1)$] $g$ satisfies the ray condition.
				\item[$(2)$] $|\ind(h)|=|\ind(th)|$, for all $h\in\smd(g)$ and $t\in\mathbb{N}$.
			\end{enumerate}
			In particular, $\Lambda$ satisfies the ray condition if and only if $|\ind(g)|=|\ind(tg)|$, for all g-vectors $g$ and $t\in\mathbb{N}$.
		\end{remark}
		\begin{definition}\label{def4000}
			Consider a g-vector $g$. We define
			\begin{itemize}
				\item $D_g:=\{h\in K_0(\proj\Lambda)\mid\exists t\in\mathbb{N}, g+th=g\oplus th\}$
				\item $D_{\mathbb{N}g}:=\bigcup_{r\in\mathbb{N}}D_{rg}$.
			\end{itemize}
		\end{definition}
		\begin{remark}\label{rem4100}
			Let $g$ be a g-vector. Then $g\in D_{\mathbb{N}g}$ if and only if $tg$ is tame for some $t\in\mathbb{N}$. In particular, if $g$ is tame, then $g\in D_g$.
		\end{remark}
		\begin{remark}\label{rem4200}
			For a g-vector $g$, the following conditions are equivalent.
			\begin{enumerate}
				\item[$(1)$] $g$ satisfies the regularity condition.
				\item[$(2)$] $h\notin D_{\mathbb{N}h}$, for all wild $h\in\smd(g)$.
			\end{enumerate}
			In particular, $\Lambda$ satisfies the regularity condition if and only if $g\notin D_{\mathbb{N}g}$ for all wild g-vectors $g$.
		\end{remark}
		\begin{proposition}\label{prop4300}
			Let $g$ be a g-vector. Then the following statements hold.
			\begin{enumerate}
				\item[$(1)$] If $g$ is wild and contained in $D_g$, then
				$|\ind(g\oplus tg)|>|\ind(t(g\oplus tg))|$
				for some $t\in\mathbb{N}$.
				\item[$(2)$] If $|\ind(tg)|<|\ind(g)|$, for some $t\in\mathbb{N}$, then there is a wild g-vector $h\in\ind(g)$ with $h\in D_{h}$.
			\end{enumerate}
		\end{proposition}
		\begin{proof}
			 (1) Let $t=\min\{r\in\mathbb{N}\mid (1+r)g=g\oplus rg\}$. One can see that $tg$ is tame and so by Remark \ref{rem3700}, $|\ind(t(g\oplus tg))|=|\ind((1+t)tg)|=|\ind(tg)|$. But $|\ind(g\oplus tg)|\ge|\ind(tg)|$. If the equality holds, then $\ind(g)\subseteq\ind(tg)$. In this case, $g$ is a direct summand of $tg$ and it contradicts the minimality of $t$. Therefore, $|\ind(g\oplus tg)|>|\ind(t(g\oplus tg))|$.
			 
			 (2) Consider  the generic decomposition $g=h_1\oplus\cdots\oplus h_r\oplus g_1\oplus\cdots\oplus g_s$, where $h_i$ is tame and $g_i$ is wild. Suppose that for each $1\le i\le s$, there exits $g'_i\in\ind(tg_i)\setminus\ind(tg-tg_i)$. Then
			 $$h'=h_1\oplus\cdots\oplus h_r\oplus g'_1\oplus\cdots\oplus g'_s\in\smd(tg).$$
			 So $|\ind(g)|=|\ind(h')|\le|\ind(tg)|$. It contradicts the assumption of $(2)$. Hence, there exists $1\le i\le s$ such that $\ind(tg_i)\subseteq\ind(tg-tg_i)$. So $g_i\in D_{g_i}$.
		\end{proof}
		\begin{definition}\label{def4400}
			Let $g$ be a g-vector. Then $g$ satisfies the \textit{non-decreasing condition}, if $h\notin D_h$, for all wild $h\in\ind(g)$.
			We say $\Lambda$ satisfies the non-decreasing condition, if all g-vectors satisfy the non-decreasing condition.
		\end{definition}
		\begin{proposition}\label{prop4450}
			The ray condition implies the regularity condition and the regularity condition implies the non-decreasing condition. In particular, tame g-vectors satisfy all of the three conditions.
		\end{proposition}
		\begin{proof}
			Consider a g-vector $g$. Let $h\in\ind(g)$ be a wild g-vector. If $th$ is tame, then $2th=th\oplus th$. Hence, $2th$ is not generically indecomposable. Therefore, the ray condition implies the regularity condition. The other parts are also clear.
		\end{proof}
		\begin{question}\label{ques4500}
			Is there any finite dimensional algebra that does not satisfy the non-decreasing condition?
		\end{question}
		Proposition \ref{prop4300} and the following observation are our main reasons to call this condition ``non-decreasing''.
		\begin{theorem}\label{prop4600}
			$\Lambda$ satisfies the non-decreasing condition if and only if
			$$|\ind(g)|\le|\ind(tg)|,$$
			for all g-vectors $g$ and $t\in\mathbb{N}$. Indeed, for a g-vector $g$, the following statements are equivalent.
			\begin{enumerate}
				\item[$(1)$] $tg$ satisfies the non-decreasing condition, for all $t\in\mathbb{N}$.
				\item[$(2)$] $|\ind(h)|\le|\ind(th)|$, for all $t\in\mathbb{N}$ and wild $h\in\add(g)$.
			\end{enumerate}
		\end{theorem}
		\begin{proof}
			It is a direct consequence of Proposition \ref{prop4300}.
		\end{proof}
		\begin{theorem}\label{thm4700}
			Let $g$ be a g-vector. If $g$ satisfies the non-decreasing condition, then $\ind(g)$ is linearly independent and thus $|\ind(g)|\le n$. Furthermore, the following statements are equivalent.
			\begin{enumerate}
				\item[$(1)$]  $\ind(tg)$ is linearly independent, for all $t\in\mathbb{N}$.
				\item[$(2)$] $tg$ satisfies the non-decreasing condition, for all $t\in\mathbb{N}$.
			\end{enumerate}
			In particular, $\Lambda$ satisfies the non-decreasing condition if and only if $\ind(g)$ is linearly independent, for each g-vector $g$.
		\end{theorem}
		\begin{proof}
			$(2)\Rightarrow (1):$
			Consider the generic decomposition $g=g_1\oplus g_2\oplus \cdots\oplus g_s$. We may assume that $g_i\neq g_j$, for all $1\le i\neq j\le s$.
			Suppose that there exist non-empty subsets $I,J$ of $\{1,2, \cdots, s\}$ with an empty intersection, such that
			$$\bigoplus_{i\in I}a_i g_i=\bigoplus_{j\in J}b_j g_j$$
			for some $a_i,b_j\in\mathbb{N}$. Since the right hand side is in $D_{g_i}$ for $i\in I$, we conclude that $g_i\in D_{g_i}$. So $g_i$ is tame. Similarly, for all $l\in I\cup J$, $g_l$ is tame. Therefore, by the uniqueness of generic decompositions, we deduce that $\{g_i\mid i\in I\}=\{g_j\mid j\in J\}$. It contradicts $I\cap J=\emptyset$. Thus $\ind(g)$ is linearly independent.
			
			$(1)\Rightarrow (2):$
			Suppose that $(2)$ does not hold. We may assume that there is a wild g-vector $h\in\ind(g)$ with $h\in D_h$. Set
			$$t=\min\{t'\in\mathbb{N}\mid h+t'h=h\oplus t'h\}.$$
			Because of the minimality of $t$, one can check that $h$ is not a direct summand of $th$. Consider a generic decomposition $th=h_1\oplus h_2\oplus \cdots\oplus h_s$. Then $h+th=h\oplus h_1\oplus\cdots\oplus h_s$ is a generic decomposition. But $\ind(h+th)=\{h,h_1,\cdots,h_s\}$ is not linearly independent. Therefore, $\ind((1+t)g)$ is not linearly independent. It is a contradiction.
		\end{proof}
		\begin{problem}
			Let $g$ be a g-vector. Is it possible to remove ``t'' from the statements $(1)$ and $(2)$ in Theorem \ref{thm4700}?\footnote{This problem was posed by Toshiya Yurikusa in our email-discussion.}
		\end{problem}
	\section{Generically $\tau$-reduced components}
			In this section, we work on direct sums of generically $\tau$-reduced components and especially we study the relation between generic decompositions of g-vectors and decompositions of generically $\tau$-reduced components. At the first step, we need to look into Crawley-Boevey-Schroer’s decomposition theorem \cite[Theorem 1.1]{CBSc02}. Then, we will peruse its $\tau$-reduced version proved in \cite{CILFS14}.
		
		\begin{proposition}[{\cite[Corollary 1.4]{Ga75}}]\label{prop4760}
			Consider $\mathbf{d}\in\mathbb{N}$. Assume that for a dimension vector $d=(d_1,\cdots,d_n)$, $\mathbf{d}=\Sigma_{i=1}^nd_i$. Then $\rep_d(\Lambda)$ is a connected component of $\rep_{\mathbf{d}}(\Lambda)$. Moreover, all connected components of $\rep_{\mathbf{d}}(\Lambda)$ arise in this way.
		\end{proposition}
		Therefore, rather than varieties of representations of the same dimension, we work on varieties of representations with the same dimension vectors.
		\begin{namedthm}{Decomposition of irreducible components}\label{defthm4800}
			Consider dimension vectors $d_1, \cdots, d_s$.
			For each $1\le i\le s$, let $\mathcal{Z}_i\subseteq\rep_{d_i}(\Lambda)$ be a $\GL_{d_i}(k)$-stable subset. Consider the map
			$$\GL_d(k)\times\mathcal{Z}_1\times\cdots\times\mathcal{Z}_s \rightarrow\rep_d(\Lambda),$$
			$$(\mathfrak{g},M_1,\cdots,M_s)\mapsto\mathfrak{g}\cdot(M_1\oplus M_2\oplus\cdots\oplus M_s)$$
			where $d=d_1+d_2+\cdots+d_s$.
			The image of the map is denoted by
			\[\mathcal{Z}_1\oplus\mathcal{Z}_2\oplus \cdots\oplus\mathcal{Z}_s.\]
			If $\mathcal{Z}_i$ is an irreducible locally closed subset, for each $1\le i\le s$, then the Zariski-closure $\overline{\mathcal{Z}_1\oplus\mathcal{Z}_2\oplus \cdots\oplus\mathcal{Z}_s}$ is an irreducible closed subset of $\rep_d(\Lambda)$.
			Moreover, if $\mathcal{Z}_i$ is an irreducible component, for each $1\le i\le s$, then the following statements are equivalent.
			\begin{enumerate}
				\item[$(1)$] $\overline{\mathcal{Z}_1\oplus\mathcal{Z}_2\oplus \cdots\oplus\mathcal{Z}_s}$ is an irreducible component of $\rep_d(\Lambda)$
				\item[$(2)$] $\min\ext_\Lambda^1(\mathcal{Z}_i,\mathcal{Z}_j)=0,$
				for all $1\le i\neq j\le s$
			\end{enumerate}
			An irreducible component $\mathcal{Z}$ is said to be \textit{generically indecomposable}, if a general module in $\mathcal{Z}$ is indecomposable.
			Each irreducible component $\mathcal{Z}$ can be written as a direct sum of generically indecomposable irreducible components $\mathcal{Z}=\overline{\mathcal{Z}_1\oplus\mathcal{Z}_2\oplus \cdots\oplus\mathcal{Z}_s}$, in a unique way. In this case, we set $|\mathcal{Z}|=|\{\mathcal{Z}_1,\cdots,\mathcal{Z}_s\}|$.
		\end{namedthm}
		\begin{lemma}\label{lem4900}
			Consider dimension vectors $d_1, \cdots, d_s$ and an irreducible component $\mathcal{Z}_i$ of $\rep_{d_i}(\Lambda)$, for each $1\le i\le s$. Assume that $d=d_1+d_2+\cdots+d_s$ and
			\[\mathcal{Z}=\overline{\mathcal{Z}_1\oplus\mathcal{Z}_2\oplus \cdots\oplus\mathcal{Z}_s}\]
			is an irreducible component of $\rep_{d}(\Lambda)$. Let $\tilde{\mathcal{U}_i}$ be a non-empty open (dense) subset of $\mathcal{Z}_i$, for each $1\le i\le s$. Then there exists a non-empty open (dense) subset $\tilde{\mathcal{U}}$ of $\mathcal{Z}$ such that $\tilde{\mathcal{U}}\subseteq\tilde{\mathcal{U}}_1\oplus\tilde{\mathcal{U}_2}\oplus\cdots\oplus\tilde{\mathcal{U}_s}$.Therefore, a general module in $\mathcal{Z}$ can be written as a direct sum of general modules in $\mathcal{Z}_i$ for $1\le i\le s$.
		\end{lemma}
		\begin{proof}
			As $\GL_d(k)\times\tilde{\mathcal{U}_1}\times \cdots\times\tilde{\mathcal{U}_s}$ is open and dense in $\GL_d(k)\times\mathcal{Z}_1\times\cdots\times\mathcal{Z}_s$, by Remark \ref{rem920}, the closure of its image is $\mathcal{Z}$.
			Since $\tilde{\mathcal{U}_1}\oplus\cdots\oplus\tilde{\mathcal{U}_s}$ is constructible and its closure is irreducible, by Theorem \ref{thm1200}, there exists a non-empty open (dense) subset $\tilde{\mathcal{U}}$ of $\mathcal{Z}$ such that $\tilde{\mathcal{U}}\subseteq\tilde{\mathcal{U}_1}\oplus\cdots\oplus\tilde{\mathcal{U}_s}$
		\end{proof}
		\begin{lemma}[{\cite[Corollary 2.2]{Pla13}}]\label{lem5000}
			Consider a dimension vector $d$ and an irreducible component $\mathcal{Z}\subseteq\rep_d(\Lambda)$. There exist a non-empty open (dense) subset $\tilde{\mathcal{U}}$ of $\mathcal{Z}$ and $P^{-1},P^0\in\proj\Lambda$ such that every $M\in\tilde{\mathcal{U}}$ has a minimal projective presentation of the form
			$$P^{-1}\rightarrow P^0 \rightarrow M \rightarrow 0.$$
		\end{lemma}
		By the above lemma, one can assign a g-vector to each irreducible component as follows.
		\begin{definition}\label{def5100}
			Let $\mathcal{Z}$ be an irreducible component of $\rep_d(\Lambda)$ for a dimension vector $d$. Then the g-vector of a general element of $\mathcal{Z}$ is said to be the g-vector of $\mathcal{Z}$ and denoted by $g^{\mathcal{Z}}$.
		\end{definition}
		\begin{lemma}[{\cite[Lemma 5.12]{CILFS14}}]\label{lem5200}
			Consider irreducible components $\mathcal{Z}$, $\mathcal{Z}_1,\cdots, \mathcal{Z}_s$.
			If $\mathcal{Z}=\overline{\mathcal{Z}_1\oplus\mathcal{Z}_2\oplus \cdots\oplus\mathcal{Z}_s}$, then
			$g^{\mathcal{Z}}=g^{\mathcal{Z}_1}+\cdots+g^{\mathcal{Z}_s}$.
		\end{lemma}
		\begin{proof}
			By Lemma \ref{lem5000}, for each $1\le i\le s$, there is a non-empty open (dense) subset $\tilde{\mathcal{U}_i}$ of $\mathcal{Z}_i$ such that for $M_i\in\tilde{\mathcal{U}_i}$, $g^{M_i}=g^{\mathcal{Z}_i}$. So by Lemma \ref{lem4900} and Lemma \ref{lem5000}, there exists a non-empty open (dense) subset $\tilde{\mathcal{U}}$ of $\mathcal{Z}$ such that $\tilde{\mathcal{U}}\subseteq\tilde{\mathcal{U}}_1\oplus\tilde{\mathcal{U}_2}\oplus\cdots\oplus\tilde{\mathcal{U}_s}$ and for each $M\in\tilde{\mathcal{U}}$, $g^{M}=g^{\mathcal{Z}}$. Therefore, $g^{\mathcal{Z}}=g^{\mathcal{Z}_1}+\cdots+g^{\mathcal{Z}_s}$.
		\end{proof}
			By the above assumption, the decomposition of a given irreducible component does not necessarily provide the generic decomposition of its g-vector. Later, we will see that
			$$g^{\mathcal{Z}}=g^{\mathcal{Z}_1}\oplus\cdots\oplus g^{\mathcal{Z}_s},$$
			if and only if $\mathcal{Z}$ is a generically $\tau$-reduced component.
		\begin{remark}\label{rem5300}
			Let $d$ be a dimension vector and $\mathcal{Z}$ be an irreducible component of $\rep_d(\Lambda)$. It is well known that $\mathcal{Z}$ is $\GL_d(k)$-stable. Additionally, by the Voigt’s Isomorphism and the Auslander-Reiten Duality, for each $M\in\mathcal{Z}$, we have
			$$\codim_{\mathcal{Z}}\mathcal{O}_M\le\ext_\Lambda^1(M,M)\le\hom_\Lambda(M,\tau M).$$
		\end{remark}
		\begin{definition}\label{def5400}
			Consider a dimension vector $d$. An irreducible component $\mathcal{Z}$ of $\rep_d(\Lambda)$ is said to be \textit{generically $\tau$-reduced}, if for a general $M\in\mathcal{Z}$,
			$$\hom_{\Lambda}(M,\tau M)=\codim_{\mathcal{Z}}\mathcal{O}_M.$$
		\end{definition}
		\begin{theorem}[{\cite[Theorem 1.2]{Pla13}}]\label{thm5500}
			Let $g$ be a g-vector. There is a non-empty open (dense) subset $\mathcal{U}$ of $\Hom_{\Lambda}(g)$ and a dimension vector $d(g)$ satisfying the following conditions.
			\begin{enumerate}
				\item[$(1)$] For any $a\in\mathcal{U}$, $[\Coker(a)]=d(g)$, and the map
				$$\Coker:\mathcal{U}\rightarrow\rep_{d(g)}(\Lambda)$$
				is a morphism of varieties.
				\item[$(2)$] $\mathcal{Z}_g:=\overline{\bigcup_{a\in\mathcal{U}}\mathcal{O}_{\Coker(a)}}$ is a generically $\tau$-reduced component of $\rep_{d(g)}(\Lambda)$.
			\end{enumerate}
			Moreover, we have a surjective map
			$$\mathcal{Z}_{-}:K_0(\proj\Lambda)\rightarrow\{\text{generically $\tau$-reduced components of $\rep(\Lambda)$}\}$$
			such that $\mathcal{Z}_g=\mathcal{Z}_h$ if and only if their generic decompositions are the same, when one forgets negative direct summands.
		\end{theorem}
		\begin{corollary}\label{cor5600}
			Let $g$ be a g-vector with no negative direct summand.
			Then the open (dense) subset $\mathcal{U}$ in Theorem \ref{thm5500} can be considered such that it satisfies in the following extra condition.
			\begin{enumerate}
				\item[$(3)$] All $a\in\mathcal{U}$ are minimal projective presentations.
			\end{enumerate}
			Moreover, any general $M$ in $\mathcal{Z}_g$ is in $\mathcal{O}_{\Coker(a)}$, for some $a\in\mathcal{U}$.
		\end{corollary}
		\begin{proof}
			Since for all open dense $\mathcal{U}'\subseteq\mathcal{U}$ we have 
			$$\overline{\bigcup_{a\in\mathcal{U}'}\mathcal{O}_{\Coker(a)}}=\overline{\bigcup_{a\in\mathcal{U}}\mathcal{O}_{\Coker(a)}},$$
			the statement follows from the next lemma \ref{lem5700}. For the last part, note that by Theorem \ref{thm1200}, $\bigcup_{a\in\mathcal{U}}\mathcal{O}_{\Coker(a)}$ is constructible and dense in the irreducible component $\mathcal{Z}_g$. So it is immediately deduced by Lemma \ref{lem1100}.
		\end{proof}
		\begin{lemma}[{\cite[Corollary 2.8]{Pla13}}]\label{lem5700}
			Let $g$ be a g-vector. If the generic decomposition of $g$ has no negative
			term, then a general element of $\Hom_{\Lambda}(g)$ is a minimal projective presentation.
		\end{lemma}
		\begin{proof}
			One can check that if a non-zero indecomposable projective presentation is minimal. Thus by Remark \ref{rem3500}, a general morphism of $\Hom_{\Lambda}(g)$ is a minimal projective presentation.
		\end{proof}
		\begin{lemma}\label{lem5800}
			Let $g$ be a g-vector with no negative direct summand. Then $\mathcal{Z}_g$ has an open (dense) subset $\tilde{\mathcal{U}}$ such that any $M\in\tilde{\mathcal{U}}$ admits a minimal projective presentation of the form
			\begin{center}
				\begin{tikzcd}[cramped, sep=small] 
					P^{g-} \arrow[r] & P^{g+} \arrow[r] & M \arrow[r] & 0.
				\end{tikzcd}
			\end{center}
			Therefore, by definition, $g^{\mathcal{Z}_g}=g$ and so the map
			$$g^{-}:\{\text{generically $\tau$-reduced components of $\rep(\Lambda)$}\}\rightarrow\mathbb{Z}^{n}$$
			gives us a bijection between the set of generically $\tau$-reduced components and g-vectors with non-negative\footnote{We say $g$ is \textit{non-negative} (\textit{non--positive}, respectively), if $[P^{g+}]\neq 0$ ($[P^{g-}]\neq 0$, respectively).} direct summands.
			Moreover, $g$ is a direct summand of any g-vector $h$ with $\mathcal{Z}_h=\mathcal{Z}_g$.
		\end{lemma}
		\begin{proof}
			It follows from \cite[Corollary 2.17]{Pla13}.
		\end{proof}
		\begin{remark}\label{rem5810}
			In the case that $g$ has no negative direct summand, it was shown in \cite[Section 2.6]{Pla13} that
			$$\Hom_{\Lambda}(g)_{max}:=\{a\in\Hom_{\Lambda}(g)\mid\text{$a$ is of maximal rank}\}$$
			is equal to
			$$\Hom_{\Lambda}(g)_{d(g)}:=\{a\in\Hom_{\Lambda}(g)\mid[\Coker(a)]=d(g)\}.$$
			As the first one is an open dense subset of $\Hom_{\Lambda}(g)$, then so is the second one.
		\end{remark}
		We spend the rest of this section to prove the following conjecture under the non-decreasing condition. Indeed, we will show that the non-decreasing condition on a given g-vector $g$ guarantees that $|\mathcal{Z}_g|$ is lower than or equal to $n$, \ref{thm6500}. Afterwards, we apply results of \cite{As21} regarding the wall and chamber structure of $\Lambda$ to prove the second statement, when the non-decreasing condition is satisfied.
		\begin{lemma}[{\cite[Corollary 1.4]{GLFS23}}]\label{lem6050}
			For positive integers $d$ and $d'$, the following functions are upper semi-continuous.
			\[\begin{array}{ll}
				\hom_{\Lambda}(-,\tau ?) :\rep_d(\Lambda)\times\rep_{d'}(\Lambda)\rightarrow\mathbb{Z},& \hom_{\Lambda}(-,\tau -):\rep_d(\Lambda)\rightarrow\mathbb{Z}
			\end{array}\]
		\end{lemma}
		\begin{conjecture}[{\cite[Conjecture 6.3]{CILFS14}}]\label{con5900}
			Let $\mathcal{Z}=\overline{\mathcal{Z}_1\oplus\cdots\oplus\mathcal{Z}_s}$ be a generically $\tau$-reduced component. Then the following statements hold.
			\begin{itemize}
				\item[$(1)$] $|\mathcal{Z}|\le|\Lambda|$.
				\item[$(2)$] $|\mathcal{Z}|=|\Lambda|$ if and only if for all $1\le i\le s$,
				\[\min\{\hom_{\Lambda}(X,\tau X)\mid X\in\mathcal{Z}_i\}=0,\]
				and $\mathcal{Z}$ is maximal, that is, if $\mathcal{Z}'$ and $\overline{\mathcal{Z}\oplus\mathcal{Z}'}$ are generically $\tau$-reduced components, then $|\mathcal{Z}|=|\overline{\mathcal{Z}\oplus\mathcal{Z}'}|$.
			\end{itemize}
		\end{conjecture}
		Now, we study the link between generic decompositions of g-vectors and decompositions of generically $\tau$-reduced components. It is the key to achieving our goal.
		\begin{lemma}[{\cite[Lemma 2.6]{Pla13}}]\label{lem6000}
			For projective presentations
			\begin{center}
				\begin{tikzcd}[cramped, sep=small] 
					P^{-1} \arrow[r, "a"] &P^0 \arrow[r] &M \arrow[r] &0
				\end{tikzcd}
				and
				\begin{tikzcd}[cramped, sep=small] 
					P'^{-1} \arrow[r, "a'"] &P'^0 \arrow[r] &M' \arrow[r] &0,
				\end{tikzcd}
			\end{center}
			$e(a,a')\ge\hom_{\Lambda}(M',\tau M)$ and the equality holds if $a$ is minimal.
		\end{lemma}
		The next lemma which plays an important role in the proofs of Theorem \ref{thm6200} and Theorem \ref{thm6800} was independently proved by Calvin Pfeifer in a similar way \cite[Lemma 4.6]{Pfi23}.
		\begin{lemma}\label{lem6100}
			Let $g$ and $h$ be g-vectors with no negative direct summand. Then
			$$e(g,h)=\min\hom_{\Lambda}(\mathcal{Z}_{h},\tau\mathcal{Z}_{g}).$$
		\end{lemma}
		\begin{proof}
			We set $m_0=\min\hom_{\Lambda}(\mathcal{Z}_{h},\tau\mathcal{Z}_{g})$.
			By Lemma \ref{lem5800}, there exist non-empty open (dense) subsets $\tilde{\mathcal{U}_g}\subseteq\mathcal{Z}_g$ and $\tilde{\mathcal{U}_h}\subseteq\mathcal{Z}_h$ with the property that each $M\in\tilde{\mathcal{U}_g}$ and $N\in\tilde{\mathcal{U}_h}$ have minimal projective presentations of the form
			\begin{center}
				\begin{tikzcd}[cramped, sep=small] 
					P^{g-} \arrow[r, "a"] & P^{g+} \arrow[r] & M \arrow[r] & 0, &
					P^{h-} \arrow[r, "b"] & P^{h+} \arrow[r] & N \arrow[r] & 0.
				\end{tikzcd}
			\end{center}
			By Lemma \ref{lem6050}, the function $\hom_{\Lambda}(-,\tau ?):\mathcal{Z}_h\times\mathcal{Z}_g\rightarrow\mathbb{Z}$ is upper semi-continuous. Thus there exists a non-empty open (dense) subset $\tilde{\mathcal{V}}\subseteq\tilde{\mathcal{U}_h}\times\tilde{\mathcal{U}_g}$ such that for all $(N,M)\in\tilde{\mathcal{V}}$,
			$\hom_{\Lambda}(N,\tau M)=m_0$.
			So Lemma \ref{lem6000} implies that $e(a,b)=m_0$ and thus $e(g,h)\le m_0$.\\
			On the other hand, as the function $e(-,?):\Hom_{\Lambda}(g)\times\Hom_{\Lambda}(h)\rightarrow\mathbb{Z}$ is upper semi-continuous,
			$$\mathcal{V}=\{(a,b)\in\Hom_{\Lambda}(g)\times\Hom_{\Lambda}(h)\mid e(a,b)=e(g,h)\}$$
			is open and dense. Whereas $\Hom_{\Lambda}(g)_{d(g)}\subseteq\Hom_{\Lambda}(g)$ and $\Hom_{\Lambda}(h)_{d(h)}\subseteq\Hom_{\Lambda}(h)$ are non-empty open (dense) subsets, by Corollary \ref{cor5600}(3),
			$$\mathcal{V}'=\{(a,b)\in\mathcal{V}\mid\text{$a\in\Hom_{\Lambda}(g)_{d(g)}$ and $b\in\Hom_{\Lambda}(h)_{d(h)}$ are minimal}\}$$
			is open (and dense), as well. Hence, by Lemma \ref{lem6000}, for each $(a,b)\in\mathcal{V}'$, we have
			$$e(g,h)=e(a,b)=\hom_{\Lambda}(\Coker(b),\tau\Coker(a))\ge m_0.$$
			\end{proof}
		\begin{theorem}\label{thm6200}
			Let $\{\mathcal{Z}_1, \mathcal{Z}_2,\cdots, \mathcal{Z}_s\}$ be a set of irreducible components of $\rep\Lambda$. Then the following statements are equivalent.
			\begin{enumerate}
				\item[$(1)$] $\mathcal{Z}=\overline{\mathcal{Z}_{1}\oplus\mathcal{Z}_{2}\oplus\cdots\oplus \mathcal{Z}_{s}}$ is a generically $\tau$-reduced component.
				\item[$(2)$] Each $\mathcal{Z}_i$ is a generically $\tau$-reduced component and
				$$\min\hom_{\Lambda}(\mathcal{Z}_i,\tau\mathcal{Z}_j)=0,$$
				for all $1\le i\neq j\le s$.
				\item[$(3)$] Each $\mathcal{Z}_i$ is a generically $\tau$-reduced component and
				$$e(g^{\mathcal{Z}_i},g^{\mathcal{Z}_j})=0,$$
				for all $1\le i\neq j\le s$.
			\end{enumerate}
			In this case, $g^{\mathcal{Z}}=g^{\mathcal{Z}_1}\oplus\cdots\oplus g^{\mathcal{Z}_s}$ is the generic decomposition of $g^{\mathcal{Z}}$.
		\end{theorem}
		\begin{proof}
			It is an immediate consequence of \cite[Theorem 1.3]{CILFS14}, and Lemma \ref{lem6100} together with Lemma \ref{lem5200}.
		\end{proof}
		\begin{corollary}\label{cor6300}
			For a g-vector $g$ with no negative direct summand, the following statements are equivalent.
			\begin{itemize}
				\item[$(1)$] $g=g_1\oplus g_2 \oplus\cdots\oplus g_s$ for some $g_i\in K_0(\proj\Lambda)$, $1\le i\le s$.
				\item[$(2)$] $\mathcal{Z}_g=\overline{\mathcal{Z}_{g_1}\oplus\mathcal{Z}_{g_2}\oplus\cdots\oplus \mathcal{Z}_{g_s}}.$
			\end{itemize}
			Moreover, a generically $\tau$-reduced component $\mathcal{Z}$ is generically indecomposable if and only if $g^{\mathcal{Z}}$ is generically indecomposable.
		\end{corollary}
		\begin{proof}
			It follows from Theorem \ref{thm6200} and the uniqueness of generic decompositions.
		\end{proof}
		\begin{remark}\label{rem6400}
			Let $g$ be a g-vector and $g=g_1\oplus g_2 \oplus\cdots\oplus g_s$ a generic decomposition. Then
			$$d(g)=\sum_{i=1}^{s}d(g_i).$$
		\end{remark}
		Theorem \ref{thm6200} leads us to propose a generalization for \cite[Theorem 2.9]{GLFS23} or \cite[Theorem 6.1]{CILFS14}. Note that due to Lemma \ref{lem6100}, for a generically $\tau$-reduced component $\mathcal{Z}$, $\min\hom_{\Lambda}(\mathcal{Z},\tau\mathcal{Z})=0$ if and only if $g^{\mathcal{Z}}$ is tame.
		\begin{theorem}\label{thm6500}
			Let $g$ be an arbitrary g-vector.
			Then $|\mathcal{Z}_g|\le|\ind(g)|$ and the equality holds if and only if $g$ has no negative direct summand. In particular, if $g$ satisfies the non-decreasing condition, then $|\mathcal{Z}_g|\le|\Lambda|$.
		\end{theorem}
		\begin{proof}
			It is a direct consequence of Theorem \ref{thm6200} and Theorem \ref{thm4700}.
		\end{proof}
		According to Remark \ref{rem3750}, we provide the following consequence which together with Theorem \ref{thm6200} proves Conjecture \ref{con5900}($2$), under the non-decreasing condition.
		\begin{corollary}\label{cor6650}
			Let $\mathcal{Z}$ be a generically $\tau$-reduced component. Then the following statements are equivalent.
			\begin{itemize}
				\item[$(1)$] $g^\mathcal{Z}\in\mathcal{C}^{\circ}(\mathrm{P})$ for a $2$-term silting complex $\mathrm{P}\in K^b(\proj\Lambda)$.
				\item[$(2)$] $|\mathcal{Z}|=|\Lambda|$ and $g^\mathcal{Z}$ satisfies the non-decreasing condition.
			\end{itemize}
		\end{corollary}
	\section{Numerical criteria for tameness}
		We use the following lemma to obtain some criteria for tameness of g-vectors. Then we provide some examples and computations.
		\begin{lemma}[{\cite[Theorem 1.4]{AuRe85}}]\label{lem6700}
			For modules $M$ and $N$,
			\[\langle g^M,[N]\rangle=\hom_\Lambda(M,N)-\hom_{\Lambda}(N,\tau M)\]
		\end{lemma}
		\begin{theorem}\label{thm6800}
			Let $g$ and $h$ be g-vectors with no negative direct summand. Then
			$$e(g,h)=\min\hom_{\Lambda}(\mathcal{Z}_g,\mathcal{Z}_h)-\langle g,d(h) \rangle.$$
		\end{theorem}
		\begin{proof}
			It follows from Lemma \ref{lem6100} and upper semi-continuity of the functions $e(-,?)$, $\hom_{\Lambda}(?,\tau-)$ and $\hom_{\Lambda}(-,?)$.
		\end{proof}
		\begin{theorem}\label{thm7000}
			Consider a g-vector $g$ with no negative direct summand. Then the following statements hold.
			\begin{itemize}
				\item[$(1)$] $g$ is wild if and only if
				$\min\hom_{\Lambda}(\mathcal{Z}_g,\mathcal{Z}_g)>\langle g,d(g)\rangle.$
				\item[$(2)$] $g$ is tame if and only if
				$\langle g,d(g) \rangle=\min\hom_{\Lambda}(\mathcal{Z}_g,\mathcal{Z}_g).$
			\end{itemize}
		\end{theorem}
		\begin{proof}
			It is a direct consequence of Theorem \ref{thm6800}.
		\end{proof}
		\begin{corollary}\label{cor7200}
			By the above notation, $\langle g,d(g) \rangle\ge 0$, if $g$ is tame. In particular, if $\Lambda$ is E-tame, then $\langle g,d(g) \rangle\ge 0$, for all g-vectors $g$.
		\end{corollary}
		\begin{corollary}[{\cite[Corollary 5.14]{CILFS14}}]\label{cor7300}
			Let $d$ be a dimension vector and $\mathcal{Z}\subseteq\rep_d(\Lambda)$ be a generically $\tau$-reduced component. If $\langle g^{\mathcal{Z}},d \rangle< 0$, then
			$$\min\{\hom_{\Lambda}(M,\tau M)\mid M\in\mathcal{Z}\}>0.$$
		\end{corollary}
		\begin{proof}
			Theorem \ref{thm7000} implies that $g^{\mathcal{Z}}$ is wild. Thus $\min\hom_{\Lambda}(\mathcal{Z},\tau\mathcal{Z})=e(g^{\mathcal{Z}},g^{\mathcal{Z}})>0$. Now, the statement follows from the following inequality.
			$$\min\hom_{\Lambda}(\mathcal{Z},\tau\mathcal{Z})\le\min\{\hom_{\Lambda}(M,\tau M)\mid M\in\mathcal{Z}\}.$$
		\end{proof}
		\begin{example}\label{ex7350}
			Consider an $m$-Kronecker’s quiver, for $m\ge 3$.
			\begin{center}
				\begin{tikzcd}
					1 & 2
					\arrow[shift right=2, from=1-1, to=1-2]
					\arrow[shift left=3, from=1-1, to=1-2]
					\arrow["\vdots"{marking, allow upside down}, draw=none, from=1-1, to=1-2]
				\end{tikzcd}
			\end{center}
			We have $S_{(2)}\cong P_{(2)}$ and $[P_{(1)}]=(1,m)$. Consider the g-vector $g=[P_{(1)}]-[P_{(2)}]=(1,-1)$. It is easy to compute $d(g)=[P_{(1)}]-[S_{(2)}]=(1,m-1)$. So $\langle g,d(g) \rangle=2-m<0$. Thus by Corollary \ref{cor7200}, $(1,-1)$ is a wild g-vector.
		\end{example}
		A more general statement also holds. Consider projective modules $P_0$ and $P_1$ with no non-zero direct summand in common. Assume that there is a monomorphism in $\Hom_{\Lambda}(P_1,P_0)$. Then by upper semi-continuity of the function
		$$\dim_k\Ker:\Hom_{\Lambda}(P_1,P_0)\rightarrow\mathbb{Z},$$
		the set of monomorphisms from $P_1$ to $P_0$ are open and dense. So $g=[P_0]-[P_1]$ does not admit a negative direct summand. Thus a general module in $\mathcal{Z}_g$ is of projective dimension $1$. Therefore, $d(g)=[P_0]-[P_1]\in K_0(\mod\Lambda)$.\\
		Now, assume that $P_0=P_{(1)}^{a_1}\oplus\cdots\oplus P_{(n)}^{a_n}$ and $P_1=P_{(1)}^{b_1}\oplus\cdots\oplus P_{(n)}^{b_n}$. By definition, $g=(a_1-b_1,\cdots,a_n-b_n)$. Moreover,
		we know that
		$$[P_{(i)}]=(\dim_k e_1\Lambda e_i,\cdots,\dim_k e_n\Lambda e_i),$$
		where $\{e_1,\cdots,e_n\}$ is the complete set of primitive
		orthogonal idempotents corresponding to $\{P_{(1)},\cdots,P_{(n)}\}$. So
		$$d(g)=(\sum_{i=1}^{n}(a_i-b_i)\dim_k e_1\Lambda e_i,\cdots,\sum_{i=1}^{n}(a_i-b_i)\dim_k e_n\Lambda e_i).$$
		Therefore,
		$$\langle g,d(g) \rangle=\sum_{1\le i,j \le n}(a_i-b_i)(a_j-b_j)\dim_k e_j\Lambda e_i.$$
		Note that we had assumed that $P_1$ and $P_0$ do not share any non-zero direct summands. So in the above formula, if $a_i\neq 0$, then $b_i=0$.
		Consider $A=\{1\le i \le n\mid a_i\neq 0\}$ and $B=\{1\le j \le n\mid b_j\neq 0\}$. Thus
		$$\scalemath{0.8}{\langle g,d(g) \rangle=\sum_{i,i'\in A}a_ia_{i'}\dim_k e_i\Lambda e_{i'}+\sum_{j,j'\in B}b_jb_{j'}\dim_k e_j\Lambda e_{j'}-\sum_{i\in A, j\in B}a_ib_j\dim_k (e_i\Lambda e_{j}+e_j\Lambda e_i).}$$
		Moreover, in this case, $d(tg)=td(g)$, for all $t\in\mathbb{N}$. Therefore, if $\langle g,d(g) \rangle$ is negative, then $tg$ is a wild g-vector, for all $t\in\mathbb{N}$.
		\begin{remark}\label{rem7360}
			Consider a set $A$ consisting of some vertices of $Q$ and a set $B$ consisting of some sink vertices. For $j\in B$, let
			$$A_j=\{i\in A\mid\textit{there is a path from $i$ to $j$}\}.$$
			Assume that for all $j\in B$, $A_j\neq\emptyset$ and
			$$|\bigcup_{j\in B}A_j|>|B|.$$
			Then there is a monomorphism from $\bigoplus_{j\in B}P_{(j)}$ to $\bigoplus_{i\in A}P_{(i)}$.
			Hence, for $g=\sum_{i\in A}[P_{(i)}]-\sum_{j\in B}[P_{(j)}]$,
			$$\langle g,d(g) \rangle=\sum_{i,i'\in A}\dim_k e_i\Lambda e_{i'}-\sum_{i\in A, j\in B}\dim_k e_j\Lambda e_{i}+|B|.$$
		\end{remark}
		\begin{example}
			For each of the following quivers, let $g=[P_{(i)}]-[P_{(j)}]$. Then $\langle g,d(g) \rangle<0$ and thus it is a wild g-vectors.
			
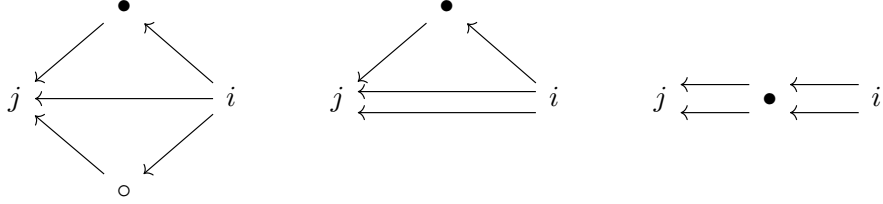
\begin{center}
			\begin{tikzcd}
				& \bullet &&& \bullet \\
				j & {} & i & j && i & j & \bullet & i \\
				& \circ
				\arrow[from=1-5, to=2-4]
				\arrow[shift left=2, from=2-6, to=2-4]
				\arrow[from=2-6, to=1-5]
				\arrow[shift right, from=2-6, to=2-4]
				\arrow[from=2-3, to=1-2]
				\arrow[from=1-2, to=2-1]
				\arrow[from=2-3, to=3-2]
				\arrow[from=3-2, to=2-1]
				\arrow[from=2-3, to=2-1]
				\arrow[shift left=2, from=2-9, to=2-8]
				\arrow[shift left=2, from=2-8, to=2-7]
				\arrow[shift right=2, from=2-9, to=2-8]
				\arrow[shift right=2, from=2-8, to=2-7]
			\end{tikzcd}
			\end{center}
			The same also holds for each of the following bound quivers.
			
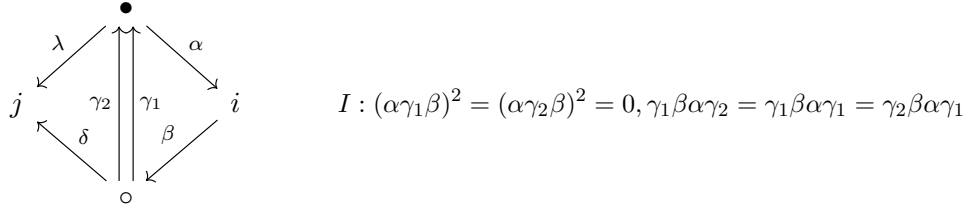
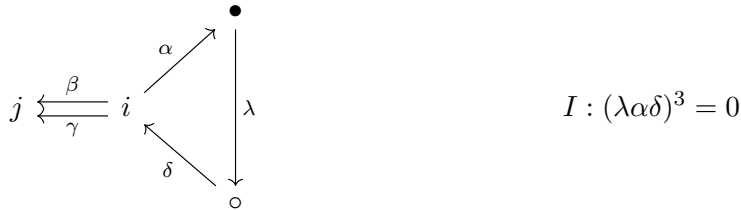
\begin{center}
			\begin{tikzcd}
					& \bullet \\
					j && i & {\scalemath{0.9}{I:(\alpha\gamma_1\beta)^2=(\alpha\gamma_2\beta)^2=0, \gamma_1\beta\alpha\gamma_2=\gamma_1\beta\alpha\gamma_1=\gamma_2\beta\alpha\gamma_1}} \\
					& \circ \\
					& {} & \bullet \\
					j & i && {I:(\lambda\alpha\delta)^3=0} \\
					&& \circ
					\arrow["\lambda"', from=1-2, to=2-1]
					\arrow["{\gamma_1}"', shift right, from=3-2, to=1-2]
					\arrow["\beta"', from=2-3, to=3-2]
					\arrow["\delta"', from=3-2, to=2-1]
					\arrow["\alpha", from=1-2, to=2-3]
					\arrow["{\gamma_2}", shift left, from=3-2, to=1-2]
					\arrow["\alpha", from=5-2, to=4-3]
					\arrow["\lambda", from=4-3, to=6-3]
					\arrow["\delta", from=6-3, to=5-2]
					\arrow["\gamma", shift left, from=5-2, to=5-1]
					\arrow["\beta"', shift right, from=5-2, to=5-1]
			\end{tikzcd}
			\end{center}
		\end{example}
		One may also check the tamness of a g-vector by using the dimension of corresponding generically $\tau$-reduced component.
		\begin{lemma}\label{lem7400}
			Let $g$ be an arbitrary g-vector. Then
			$$\dim\mathcal{Z}_g=\langle d(g),d(g) \rangle -\langle g,d(g) \rangle.$$
		\end{lemma}
		\begin{proof}
			By Lemma \ref{lem6700}, for a general $M\in\mathcal{Z}_g$, we have
			$$\langle g,d(g) \rangle =\ndo_{\Lambda}(M)-\hom_{\Lambda}(M,\tau M).$$
			Since $\mathcal{Z}_g$ is generically $\tau$-reduced, $\hom_{\Lambda}(M,\tau M)=\codim_{\mathcal{Z}_g}\mathcal{O}_M$. So
			$$\langle g,d(g) \rangle =\ndo_{\Lambda}(M)-\codim_{\mathcal{Z}_g}\mathcal{O}_M=\ndo_{\Lambda}(M)-\dim\mathcal{Z}_g+\dim\mathcal{O}_M=$$
			$$\ndo_{\Lambda}(M)-\dim\mathcal{Z}_g+\dim\GL_{d(g)}(k)-\ndo_{\Lambda}(M).$$
			Therefore, $\langle g,d(g) \rangle=\langle d(g),d(g) \rangle-\dim\mathcal{Z}_g$.
		\end{proof}
		\begin{remark}\label{rem7500}
			For an arbitrary g-vector $g$, $\langle d(g),d(g) \rangle\ge\langle g,d(g) \rangle$.
		\end{remark}
		\begin{corollary}\label{cor7600}
			Let $g$ be an arbitrary g-vector with no negative direct summand. Then
			$$e(g,g)+\langle d(g),d(g) \rangle=\dim\mathcal{Z}_g+\min\hom_{\Lambda}(\mathcal{Z}_g,\mathcal{Z}_g).$$
		\end{corollary}
		\begin{proof}
			It is obtained by combining Lemma \ref{lem7400} and Theorem \ref{thm6800}.
		\end{proof}
		\begin{remark}\label{rem7700}
			Let $g$ be a tame g-vector with no negative direct summand. Then
			$$\min\hom_{\Lambda}(\mathcal{Z}_g,\mathcal{Z}_g)\le\dim\langle d(g),d(g) \rangle\ge\dim\mathcal{Z}_g.$$
		\end{remark}
	\section*{Acknowledgment}
		We would like to express our gratitude to Pierre-Guy Plamondon for his contribution in describing some results of \cite{PYK23} and \cite{Pla13} through email.
		Additionally, we are also thankful to Sota Asai who helped us to prepare Corollary \ref{cor6650} and informed us that Lemma \ref{lem6100} was independently proved by Calvin Pfeifer.
		The authors would like to thank the anonymous referee who provided useful and detailed comments on an earlier version of the manuscript.
		The research of the authors is supported by Iran National Science Foundation (INSF), Project No. 4003197.
	
	{\small\bibliographystyle{alpha}
		\bibliography{Nondecreasingcondition}}
\end{document}